\documentclass[leqno, fleqn, a4paper, 11pt]{article}

\usepackage[latin1]{inputenc}
\usepackage[T1]{fontenc}
\usepackage{lmodern}

\usepackage[english]{babel}

\usepackage{psfrag}
\usepackage{amsmath,amsfonts,amsthm,bm,bbm}
\usepackage{psfrag,graphicx,epsfig,subfigure}
\usepackage{color}
\usepackage[footnotesize]{caption}

\usepackage{enumerate}
\usepackage{array}

\usepackage{natbib}

\usepackage{hyperref}
\hypersetup{dvips, bookmarks, bookmarksnumbered, colorlinks=false,
  breaklinks, pdfborder=0 0 0}

\graphicspath{{./figures/}}

\bibpunct{(}{)}{;}{a}{,}{,}

\newcommand   \Xset  {\mathbb{X}}
\newcommand   \XX    {\Xset}
\newcommand   \RR    {\mathbb{R}}

\renewcommand \P   {\mathsf{P}}

\newcommand   \HH    {\mathcal{H}}
\newcommand   \FF    {\mathcal{F}}
\newcommand   \GG    {\mathcal{G}}
\newcommand   \Acal  {\mathcal{A}}
\newcommand   \Ical  {\mathcal{I}}
\newcommand   \Bcal  {\mathcal{B}}

\newcommand \EE       {\mathsf{E}}

\newcommand \defeq  {\mathrel{\mathop:}=}
\DeclareMathOperator \Span {{\rm span}}

\DeclareMathOperator* \argmax {argmax}

\newcommand{\dotvar}{\,\bm{\cdot}\,}

\providecommand{\abs}[1]{\lvert#1\rvert}

\providecommand{\ns}[1]{\lVert#1\rVert}
\providecommand{\bns}[1]{\big\lVert#1\big\rVert}

\newcommand \dx    {\ddiff x}

\renewcommand{\hat}{\widehat}

\newcommand \xihat {\hat{\xi}}

\newcommand \Xvn[1][]  {\underline{X\mskip -2mu}_n^{\mskip 1.5mu #1}\mskip 1mu}

\newcommand   \Shat      {\hat S}

\theoremstyle{plain}
\newtheorem{thm}{Theorem}

\newtheorem{prop}[thm]{Proposition}

\theoremstyle{definition}

\newcommand{\ISItitle}[1]{\vskip 0pt\setlength{\parindent}{0cm}\Large\textbf{#1}\normalsize\vskip 12pt}

\begin{document}
\ISItitle{Sequential search based on kriging: convergence analysis of some algorithms}
\noindent
\begin{minipage}{1.0\linewidth}
Vazquez, Emmanuel \\[1ex]
Bect, Julien \\[1.2ex]
\it
SUPELEC, Gif-sur-Yvette, France\\ 
e-mail: emmanuel.vazquez@supelec.fr, julien.bect@supelec.fr 
\end{minipage}

\normalsize\rm\setlength{\parindent}{0.7cm}\setlength{\parskip}{3pt}

\section{Introduction}

Let $\FF$ be a set of real-valued functions on a set $\XX$ and let
$S:\FF \to \GG$ be an arbitrary mapping. We consider the problem of
making inference about $S(f)$, with $f\in\FF$ unknown, from a finite set
of pointwise evaluations of $f$.  We are mainly interested in the
problems of approximation and optimization. Formally, a deterministic
algorithm to infer a quantity of interest $S(f)$ from a set of $n$
evaluations of $f$ is a pair~$\left( \Xvn, \Shat_n \right)$ consisting
of a deterministic \emph{search strategy}
\begin{equation*}
  \Xvn : f  \mapsto \Xvn(f) = (X_1(f), X_2(f),\ldots,
  X_n(f))\in\XX^n\,,
\end{equation*}
and a mapping $\Shat_n : \FF \to \GG$, such that:
\begin{enumerate}[a)]
\item $X_1(f) = x_1$, for some arbitrary $x_1 \in \XX$
\item For all $1 \leq i < n$, $X_{i+1}(f)$ depends measurably on
  $\Ical_{i}(f)$, where $\Ical_{i} = \left( \left(X_1, Z_1 \right),
    \ldots, \left( X_i, Z_i \right) \right)$, and $Z_i(f) = f(X_i(f))$,
  $1 \le i \le n$.
\item There exists a measurable function $\phi_n$ such that
  $\Shat_n=\phi_n\circ \Ical_n$.
\end{enumerate}
The algorithm~$\left( \Xvn, \Shat_n \right)$ describes a {sequence of
  decisions}, made from an increasing amount of information: for each
$i=1,\ldots, n-1$, the algorithm uses information $\Ical_{i}(f)$ to
choose the next evaluation point~$X_{i+1}(f)$. The estimator
$\Shat_n(f)$ of~$S(f)$ is the {terminal decision}. We shall denote by
$\Acal_n$ the class of all strategies $\Xvn$ that query sequentially $n$
evaluations of $f$ and also define the subclass $\Acal_n^0\subset
\Acal_n$ of non-adaptive strategies, that is, the class of all
strategies such that the $X_i$s do not depend on $f$. 

A classical approach to study the performance of a sequential strategy
is to consider the worst error of estimation on some class of
functions $\FF$
\begin{equation*}
  \epsilon_{\rm worst case}(\Xvn) \defeq \sup_{f\in\FF} L(S(f), \Shat_n(f))\,,
\end{equation*}
where $L$ is a loss function.  There are many results dealing with the
problems of function approximation and optimization in the worst case
setting. Two noticeable results concern convex and symmetric classes of
bounded functions. For such classes, from a worst-case point of view, any
strategy will behave similarly for the problem of global optimization
and that of function approximation. Moreover the use of adaptive methods can not
be justified by a worst case analysis \citep[see,~e.g.,][Propositions
1.3.2 and 1.3.3]{novak88:_deter}. These results, combined with the fact
that most optimization algorithms are adaptive, lead to think that the
worst-case setting may not be the most appropriate framework to assess
the performance of a search algorithm in practice. Indeed, it would be
also important, in practice, to know whether the loss $L(S(f),
\Shat_n(f))$ is close to, or on the contrary much smaller than
$\epsilon_{\rm worstcase}$, for ``typical'' functions $f\in\FF$ not
corresponding to worst cases. To address this question, a classical
approach is to adopt a Bayesian point of view.

In this paper, we consider methods where $f$ is seen as a sample path of
a real-valued random process $\xi$ defined on some probability
space $(\Omega, \Bcal, \P_0)$ with parameter in $\XX$. Then, $\Xvn(\xi)$
is a random sequence in~$\XX$, with the property that
$X_{n+1}(\xi)$ is measurable with respect to  the $\sigma$-algebra
generated by $\xi(X_1(\xi))$, \ldots, $\xi(X_n(\xi))$.
From a Bayesian decision-theoretic point of view, the random
process represents prior knowledge about $f$ and makes it possible to
infer a quantity of interest before evaluating the function. This point
of view has been widely explored in the domain of optimization and
computer experiments.  Under this setting, the performance of a given strategy
$\Xvn$ can be assessed by studying the average loss
\begin{equation*}
  \epsilon_{\rm average}(\Xvn) \defeq \EE\, L(S(\xi), \Shat_n(\xi))\,.
\end{equation*}
How much does adaption help on the average, and is it possible to derive
rates of decay for errors in average? In this article, we shall make a
brief review of results concerning average error bounds of Bayesian search
methods based on a random process prior. 

This article has three parts. The precise assumptions about $\xi$ are
given in~Section~\ref{sec:kriging}. Section~\ref{sec:approximation}
deals with the problem of function approximation, while
Section~\ref{sec:optim} deals with the problem of optimization.

\section{Framework}
\label{sec:kriging}

Let $\xi$ be a random process defined on a probability space~$(\Omega,
\Bcal, \P_0)$, with parameter $x \in \RR^d$.  Assume moreover that $\xi$
has a zero mean and a continuous covariance function.  The kriging
predictor of~$\xi(x)$, based on the observations $\xi(X_i(\xi))$,
$i=1,\ldots, n$, is the orthogonal projection
\begin{equation}
  \label{eq:def-krig-pred}
  \xihat_n(x) \;\defeq\; \sum_{i=1}^n \lambda^i(x;\Xvn(\xi))\, \xi(X_i(\xi))  
\end{equation}
of $\xi(x)$ onto $\Span\{\xi(X_i(\xi)),i=1,\ldots ,n\}$
in~$L^2(\Omega,\Bcal,\P_0)$. At step $n\geq 1$, given evaluation points
$\Xvn(\xi)$, the kriging coefficients $\lambda^i(x;\Xvn(\xi))$ can be
obtained by solving a system of linear equations
\citep[see,~e.g.,][]{chiles99}. Note that for any sample path
$f=\xi(\omega,\dotvar)$, $\omega\in\Omega$, the value~$\xihat_n(\omega,
x)$ is a function of $\Ical_n(f)$ only.

The mean-square error (MSE) of estimation at a fixed point $x\in\RR^d$ will be
denoted by 
\begin{equation*}
  \sigma^2_n(x) \defeq \EE \{ (\xi(x)- \hat\xi(x;\Xvn(\xi)))^2\}\,.  
\end{equation*}
It is generally not possible to compute $\sigma^{2}_n(x)$ when
$\Xvn$ is an adaptive strategy.

\noindent
{\bf Regularity assumptions.} Assume that there exists
$\Phi:\RR^d\to \RR$ such that $k(x,y)=\Phi(x-y)$, which is in
$L^2(\RR^d)$ and has a Fourier transform
\begin{equation*}
  \tilde \Phi (u) = (2\pi)^{-d/2} \int_{\RR^d} \Phi(x) e^{i(x, u)}dx
\end{equation*}
that satisfies
\begin{equation}
  \label{eq:sobolev_cond}
  c_1 (1 + \ns{u}_2^2)^{-s} \leq \tilde \Phi(u) \leq c_2 (1 +
  \ns{u}_2^2)^{-s}\,,\quad u\in\RR^d\,,
\end{equation}
with $s>d/2$ and constants $0 < c_1 \leq c_2$.  Note that the Mat\'ern
covariance with regularity parameter~$\nu$ \citep[see,~e.g.,][]{Ste99}
satisfies such a regularity assumption, with $s = \nu +
d/2$. Tensor-product covariance functions, however, never satisfy such
a condition \cite[see][chapter~7, for some results in this
case]{ritter:2000:average}.

Let~$\HH$ be the RKHS of functions generated by~$k$. Denote
by~$(\dotvar, \dotvar)_{\HH}$ the inner product of~$\HH$, and
by~$\ns{\dotvar}_{\HH}$ the corresponding norm. It is well known
\citep[see, e.g.][]{wendland05:_scatt_data_approx} that
$\HH$ is the Sobolev space
\begin{equation*}
  W_2^s(\RR^d) = \left\{
    f \in L^2(\RR^d);\; \tilde f(\dotvar) (1+
    \ns{\dotvar}_2^2)^{s/2} \in L^2(\RR^d)
  \right\} 
\end{equation*}
due to the following result.
\begin{prop}
  \label{prop:regul-assumpt}
  $\HH \subset L^2(\RR^d)$ and
  $\forall f\in\HH$, 
  \begin{equation*}
    \ns{f}_\HH^2 = \int_{\RR^d} \abs{\tilde
      f(u)}^2\, \tilde \Phi(u)^{-1}\ du\,.    
  \end{equation*}
  $\ns{f}_{\HH}^2$ is equivalent to the Sobolev norm
  \begin{equation*}
    \ns{f}^2_{W_2^s(\RR^d)} = \ns{\tilde f(\dotvar)
      \left(1+\ns{\dotvar}_2^2\right)^{s/2}}_{L^2(\RR^d)}    
  \end{equation*}
\end{prop}

\section{Approximation}
\label{sec:approximation}

\def \eMMSE {\epsilon_\textsc{mmse}}
\def \eIMSE {\epsilon_\textsc{imse}}
\def \dx {\mathrm{d}x}

We first consider the problem of approximation, with the point of view
exposed in Section~\ref{sec:kriging}. Using the notations introduced
above, the problem of approximation corresponds to considering
operators $S$ and $\Shat_{n}$ defined by $S(\xi) \defeq \xi_{\,|\XX}$
and $\Shat_n(\xi) \defeq \xihat_{n\,|\XX}\,$, with $\XX\subset \RR^d$
a compact domain with non-empty interior. For the design of computer
experiments, classical criteria for assessing the quality of a
strategy $\Xvn \in \Acal_n$ for the approximation problem are the
maximum mean-square error (MMSE)
\begin{equation*}
  \eMMSE (\Xvn) 
  \;\defeq\; \sup_{x\in\XX}\, \EE\left( \big(\xi(x) -
    \xihat_n(x)\big)^2 \right)
  \;=\;  \sup_{x\in\XX}\, \sigma^2_n(x)
\end{equation*}
and the
integrated mean-square error (IMSE) 
\begin{equation*}
  \eIMSE (\Xvn) 
  \;\defeq\; \EE\left( \ns{\xi - \xihat_n}_{L^2(\XX,\mu)}^2 \right)
  \;=\; \int_{\XX} \sigma_n(x)^2\, \mu(\dx)
\end{equation*}
\citep[see,
e.g.,][]{sacks:89:dace,currin:1991:bpdf,welch_screening_1992,
  santner:2003:dace}. These criteria correspond to $G$-optimality and
$I$-optimality in the theory of (parametric) optimal design.

As mentioned earlier, computing $\sigma_n^2(x)$ is usually not possible
in the case of adaptive sampling strategies, even for a Gaussian
process. From a theoretical point of view, however, it is important to
know if adaptive strategies can improve upon non-adaptive strategies for
the approximation problem.

\begin{prop}
  \label{prop:adapt}
  Assume that $\xi$ is a Gaussian process.  Then adaptivity does not
  help for the approximation problem, with respect to either the MMSE
  or the IMSE criterion.
\end{prop}

\begin{proof}
  For any adaptive strategy~$\Xvn$, it can be proved by induction (using
  the fact that~$X_{i+1}$ only depends on~$\mathcal{I}_i$) that, for
  each $x \in \Xset$,

  \begin{equation}
    \label{eq:adapt-trick}
    \sigma_n^2(x) \;=\; \EE\left(
      \sigma^2(x;X_1(\xi),\ldots,X_n(\xi))
    \right) \,,
  \end{equation}
  where $\sigma^2(x ;x_1,\ldots,x_n)$, $x_1,\ldots,x_n \in \XX$, denotes the MSE at~$x$ of the
  non-adaptive strategy that selects the points $x_1, \ldots,
  x_n$. Therefore, for each $x \in \Xset$,
  \begin{equation*}
    \sigma_n^2(x) \;\ge\; \min_{x_1,\,\ldots,\,x_n\, \in\, \XX}\,
    \sigma^2(x;x_1,\,\ldots,\,x_n) \,,
  \end{equation*}
  which proves the claim in the case of the MMSE criterion. Similarly,
  integrating~\eqref{eq:adapt-trick} yields
  \begin{align*}
    \int_\Xset \sigma_n^2\, d\mu
    & \;=\; \EE\left\{
      \int_\Xset \sigma^2(x;\Xvn(\xi))\, \mu(\dx)
    \right\} \\
    & \;\ge\; \min_{x_1,\,\ldots,\,x_n\, \in\, \Xset}\,
    \int_\Xset \sigma^2(x;\,x_1,\,\ldots,\,x_n)\, \mu(\dx) \,,
  \end{align*}
  which proves the claim in the case of the IMSE criterion.
\end{proof}

In the case of the IMSE criterion, Proposition~\ref{prop:adapt} can be
seen as a special case of a general result about linear
problems \citep[see, e.g.,][Chapter~7]{ritter:2000:average}. The
following proposition establishes a connection between the MMSE
criterion and the worst-case $L^{\infty}$-error of approximation in the
unit ball of~$\HH$, which will be useful to establish the optimal rate
for IMSE- and MMSE-optimal designs.

\begin{prop}
  \label{prop:mmse-unit-ball}
  Let $\HH_1$ denote the unit ball of~$\HH$.  For any non-adaptive
  strategy $\Xvn\in\Acal_n^0$, the MMSE criterion equals the squared
  worst-case $L^{\infty}$-error of approximation in~$\HH_1$ using~$\Shat_n$:
  \begin{equation*}
    \eMMSE (\Xvn) = \left(\; \sup_{f \in \HH_1}
      \ns{S(f) - \Shat_n(f)}_{L^\infty(\XX)} 
      \;\right)^2 \,.
  \end{equation*}
\end{prop}

\begin{proof}
  Let $\Xvn \in \Acal_n^0$ be a non-adaptive strategy such that
  $X_i(\xi) = x_i$, $i=1,\ldots,n$, for some arbitrary $x_i$s in
  $\XX$. Denote by $\lambda_i(x) = \lambda_i(x;\Xvn(\xi))$ the
  corresponding kriging coefficients (which do not depend on~$\xi$). Using the
  fact that the mapping $\xi(x) \mapsto k(x,\dotvar)$ extends linearly
  to an isometry from $\overline\Span\{ \xi(y),\, y\in\RR^d \}$ to
  $\HH$, we have for all $x\in\XX$
  \begin{align*}
    \sigma_n(x)  
    & \;=\; \bns{ \xi(x)  - \xihat_n(x) }_{L^2(\Omega,\Bcal,\P_0)} \\
    & \;=\;  \bns{ k(x,\dotvar) - {\sum}_i\,  \lambda^i(x)\,
      k(x_i,\dotvar) }_{\HH} \\
    & \;=\; \sup_{f \in \HH_1} \left(\, f \,,\, k(x,\dotvar) - {\sum}_i\, 
    \lambda^i(x)\, k(x_i,\dotvar) \,\right)_{\HH}\,. \\
    & \;=\; \sup_{f \in \HH_1} (f - \Shat_n f)(x) \,.
  \end{align*}
  Thus, 
  \begin{equation*}
    \sup_{x\in\XX} \sigma_n(x) 
    \;=\; \sup_{f \in \HH_1} \sup_{x\in\XX}\; (f - \Shat_n f)(x)
    \;=\; \sup_{f \in \HH_1} \bns{ f - \Shat_nf }_{L^\infty(\XX)} \,.
  \end{equation*}
\end{proof}

The following proposition summarizes known results concerning the optimal
rate of decay in the class of non-adaptive strategies for both the
IMSE criterion and the MMSE criterion. Note that, by
Proposition~\ref{prop:adapt}, this rate is also the optimal rate of
decay in the class of all adaptive strategies if~$\xi$ is a Gaussian
process.

\begin{prop}
  \label{prop:rate-mmse-crit}
  Assume that $\xi$ has a continuous covariance function satisfying
  the regularity assumptions of Section~\ref{sec:kriging}, and let $\nu =
  s - d/2 > 0$. Then there exists $C_1 > 0$ such that, for any
  $\Xvn \in \Acal_n^0$,
  \begin{equation}
    \label{eq:rate:1}
    C_1\, n^{-2\nu/d} \;\le\; \eIMSE(\Xvn) \;\le\; \mu(\Xset)\, \eMMSE(\Xvn)
  \end{equation}
  Moreover, if $\XX$ has a Lipschitz boundary and satisfies an interior
  cone condition, then there exists $C_2 > 0$ such that
  \begin{equation}
    \label{rate:2}
    \inf_{\Xvn \in \Acal_n^0}\, \eIMSE (\Xvn)
    \;\le\; \mu(\Xset)\, \inf_{\Xvn \in \Acal_n^0} \eMMSE (\Xvn) 
    \;\le\; C_2\, n^{-2\nu / d} \,.
  \end{equation}
  The optimal rate of decay is therefore $n^{-2\nu/d}$ for both
  criteria.
\end{prop}

\begin{proof}
  It is proved in \citep[][Chapter~7,
  Proposition~8]{ritter:2000:average} that there exists $C_1 > 0$ such
  that $\eIMSE(\Xvn) \;\ge\; C_1\, n^{-2\nu/d}$ in the case where
  $\Xset = [0;1]^d$. This readily proves the lower
  bound~\eqref{eq:rate:1} since any~$\Xset$ with non-empty interior
  contains an hypercube on which Ritter's result holds.

  If $\XX$ is a bounded Lipschitz domain satisfying an interior cone
  condition, then
  \citep[][Proposition~3.2]{narcowich05:_solol_bound_funct_scatt_zeros}
  there exists $c_1 > 0$ such that $\ns{S(f) -
    \Shat_n(f)}_{L^{\infty}(\XX)} \leq c_1 h_n^{s-d/2}\,
  \ns{S(f)}_{W_2^s(\XX)}$ for all $f\in \HH$, where $h_n =
  \sup_{x\in\XX} \min_{i\in \{1,\ldots,n\} } \ns{x - X_i(f)}_2$ is the
  fill distance of the non-adaptive strategy $\Xvn$ in $\XX$. Therefore
  \begin{equation*}
    \ns{S(f) - \Shat_n(f)}_{L^{\infty}(\XX)} 
    \;\leq\; c_1 h_n^\nu\, \ns{S(f)}_{W_2^s(\XX)} 
    \;\leq\; c_1 h_n^\nu\, \ns{f}_{W_2^s(\RR^d)} 
    \;\leq\; c_2 h_n^\nu\, \ns{f}_\HH
  \end{equation*}
  for some $c_2 > 0$, using the equivalence of the Sobolev
  $W_2^s(\RR^d)$ norm with the RKHS norm (see
  Section~\ref{sec:kriging}). Considering any non-adaptive space-filling 
  strategy $\Xvn$ with a fill distance $h_n = O(n^{-1/d})$ yields 
  \begin{equation*}
    \inf_{\Xvn \in \Acal_n^0}\, \sup_{f \in \HH_1} \bns{ f - \Shat_n f
    }_{L^\infty(\XX)}
    \;\le\; c_3\, n^{-\nu/d}
  \end{equation*}
  for some $c_3 > 0$ and the upper-bound~\eqref{rate:2} then follows
  from Proposition~\ref{prop:mmse-unit-ball}.
\end{proof}

Finding a non-adaptive MMSE-optimal design is a difficult non-convex
optimization problem in $nd$ dimensions. Instead of addressing directly
such a high-dimensional global optimization problem, we can use the
classical sequential non-adaptive greedy strategy $\Xvn(\dotvar)=
\left( x_1, \ldots, x_n \right)\in\XX^n$ defined by
\begin{equation}
  \label{eq:greedy-MMSE}
  x_{i+1} \;=\; \argmax_{x\in\XX}\, \sigma^2\left( 
    x; x_1,\ldots,x_{i} \right) \,, \quad 1 \le i < n \,.
\end{equation}
Of course, the strategy is suboptimal but it only involves simpler
optimization problems in~$d$ dimensions and has the advantage that it
can be stopped at any time. 
Following
\cite{dahmen10:_conver_rates_greed_algor_reduc_basis_method}, it can
be established that this greedy strategy is rate optimal.

\begin{prop}
  Assume that $\xi$ has a continuous covariance function satisfying the
  regularity assumptions of Section~\ref{sec:kriging}, and let $\nu = s
  - d/2 > 0$.  Let $\Xvn$ be the sequential strategy defined
  by~(\ref{eq:greedy-MMSE}).  Then,
  \begin{equation*}
    \eMMSE (\Xvn) = {\rm O}(n^{2\nu/d}) \,.  
  \end{equation*}  
\end{prop}

\begin{proof}
  Theorem~3.1 in
  \citet{dahmen10:_conver_rates_greed_algor_reduc_basis_method},
  applied to the compact subset $\{ \xi(x), x \in \Xset \}$ in
  $L^2(\Omega, \Bcal, \P_0)$, states that the greedy
  algorithm~\eqref{eq:greedy-MMSE} preserves polynomial rates of
  decay. The result follows from
  Proposition~\ref{prop:rate-mmse-crit}.
\end{proof}

\section{Optimization}
\label{sec:optim}

\def \eOPTwc {\epsilon_{\textsc{opt}^{\star}}}
\def \eOPTavg   {\epsilon_{\mkern3mu\overline{\mkern-3mu\textsc{opt}\mkern-3mu}\mkern3mu}}

In this section, we consider the problem of global optimization on a
compact domain $\XX\subset \RR^{d}$, which corresponds formally to
operators $S$ and $\Shat_n$ defined by $S(\xi) = \sup_{x \in \XX}
\xi(x)$ and $\Shat_n(\xi) = \max_{i\in{1,\ldots,n}} \xi(X_i(\xi))$. 

In a Bayesian setting, a classical criterion to assess
the performance of an optimization procedure is the average error
\begin{equation*}
  \eOPTavg(\Xvn) \defeq \EE(S(\xi) - \Shat_n(\xi))\,.
\end{equation*}
Although it may be not possible in the context of this article to make a
comprehensive review of known results concerning the average case in the
Gaussian case, it can be safely said however that such results are
scarce and specific.

In fact, most available results about the average-case error concern the
one-dimensional Wiener process $\xi$ on the interval $[0,1]$. Under this
setting, \cite{ritter1990approximation} shows that the average error of
the best non-adaptive optimization procedure decreases at rate
$n^{-1/2}$ \citep[extensions of this result for non-adaptive algorithms
and the $r$-fold Wiener measure can be found
in][]{wasilkowski1992average}. Under the same assumptions for $\xi$,
\cite{calvin1997average} derives the exact limiting distribution of the
error of a particular adaptive algorithm, which suggests that adaptivity
does yield a better average error for the optimization problem---the
result is that, for any $0<\delta<1$, it is possible to find an adaptive
strategy such that $n^{(1-\delta)}(S(\xi) - \Shat_n(\xi))$ converges
in distribution.

A theoretical result concerning the optimal average-error criterion
for less restrictive Gaussian priors is also available. If the
covariance of a Gaussian process $\xi$ is $\alpha$-H\"older continuous,
then \cite{grunewalder2010regret} show that  a space filling
strategy $\Xvn$ achieves
\begin{equation}
  \label{eq:grune}
  \eOPTavg(\Xvn) = {\rm O}( n^{-\alpha/(2d)} (\log n)^{1/2} )\,.
\end{equation}
Thus, under the assumptions of Section~\ref{sec:kriging}, for a Matérn
covariance with regularity parameter $\nu$, the rate of
the optimal average error of estimation of the optimum is less than
$n^{-\nu/d} (\log n)^{1/2}$ (since a Mat\'ern covariance is
$\alpha$-H\"older continuous with $\alpha = 2\nu$). Note that this
bound is not sharp in general since the optimal non-adaptive rate 
is $n^{-1/2}$ for the Brownian motion on~$[0;1]$, the covariance
function of which is $\alpha$-H\"older continuous with $\alpha = 1$.

In view of these results, we can safely say that characterizing the
average behavior of adaptive sequential optimization algorithms is still
an open (and apparently difficult) problem. At present, the only way to
draw useful conclusions about the interest of a particular optimization
algorithm is to resort to numerical simulations. Empirical studies such
as the one presented in \cite{benassi11:_robus_gauss_bayes} for instance
are therefore very useful from a practical point of view, since they
make it possible to obtain fine and sound performance assessments of any
strategy with a reasonable computational cost.

\bibliography{refs}

\include{figures}

\end{document}